\documentclass[11pt,draft]{article}
\usepackage[utf8]{inputenc}
\usepackage[english]{babel}
\usepackage{authblk}
\usepackage{pgf}
\usepackage{tikz}
\usepackage{arydshln}
\usepackage{slashbox}
\usetikzlibrary{calc}
\usetikzlibrary{arrows}
\usetikzlibrary{patterns}
\usepackage{listings}
\usepackage{enumitem}
\usepackage{graphicx} 
\usepackage{amsthm}
\usepackage{amsfonts}

\usepackage[bf]{caption} 
\usepackage{subcaption}

\usepackage{amsmath}
\usepackage{amssymb}

\usepackage[margin=1.0in]{geometry}

 \usepackage{url}

\newtheorem{theorem}{Theorem}
\newtheorem{lemma}[theorem]{Lemma}
\newtheorem{corollary}[theorem]{Corollary}
\newtheorem{proposition}[theorem]{Proposition}

\def\semicolon{\textup{;}}

\def\x{\mathbf{x}}

\def\0{\mathbf{0}}

\usepackage{thm-restate}

\usepackage{color}

\begin{document}
\nocite{*}

\title{Charting the space of chemical nut graphs}

\author[1]{Patrick~W.~Fowler}
\author[2,3,4,5]{Toma{\v z} Pisanski}
\author[2,3,4]{Nino Ba{\v s}i{\'c}}

\affil[1]{Department of Chemistry, University of Sheffield, Sheffield S3 7HF, UK}
\affil[2]{FAMNIT, University of Primorska, Koper, Slovenia}
\affil[3]{IAM, University of Primorska, Koper, Slovenia}
\affil[4]{Institute of Mathematics, Physics and Mechanics, Ljubljana, Slovenia}
\affil[5]{Faculty of Mathematics and Physics, University of Ljubljana, Ljubljana, Slovenia}

\date{\today}
\maketitle
\begin{abstract}
\noindent Molecular graphs of unsaturated carbon frameworks
or hydrocarbons pruned of hydrogen atoms,
are chemical graphs. 
A \emph{chemical graph} is a connected simple graph of maximum degree $3$ or less. 
A \emph{nut graph} is a connected simple graph with a singular adjacency matrix that has one zero eigenvalue 
and a non-trivial kernel eigenvector 
without zero entries. 
Nut graphs have no vertices of degree $1$: they are \emph{leafless}.
The intersection of these two sets, the \emph{chemical nut graphs}, is of interest in applications in chemistry 
and molecular physics, 
corresponding to structures with fully distributed radical reactivity and omniconducting behaviour 
at the Fermi level. 
A chemical nut graph consists of $v_2 \ge 0$ vertices of degree $2$ and an even number,
$v_3 > 0$, of vertices of degree $3$.
With the aid of systematic local constructions that produce larger nut graphs from smaller, 
the combinations $( v_3, v_2)$ corresponding to 
realisable chemical nut graphs are characterised. 
Apart from a finite set of small cases, and two simply defined infinite series, 
all combinations $(v_3, v_2 )$ with even values of $v_3  > 0$ are realisable as chemical nut graphs. 
Of these combinations, only
$(20,0)$ cannot be realised by a planar chemical nut graph.
The main result characterises the ranges of edge counts for chemical
nut graphs of all orders $n$. 
\end{abstract}


{\bf Math.\ Subj.\ Class.\ (2020):} 05C50, 05C92.

\section{Introduction}
Nut graphs constitute an important subclass of graphs. 
They were introduced and first studied in a series of papers by Sciriha and co-workers
\cite{GutmanSciriha-MaxSing,Sc1998,ScOnRnkGr99,Sc2007,Sc2008,Prop23}.
A \emph{nut graph} has a singular adjacency matrix, the spectrum of which contains exactly one zero eigenvalue, 
with a corresponding non-trivial kernel eigenvector
that  has no zero entries.  
In chemistry, this property has significant consequences for the modelling of carbon frameworks 
at molecular and nano scales.
The graphs that are possible models for unsaturated carbon frameworks are the chemical graphs.
A \emph{chemical graph} is simple (without multiple edges or loops), connected, and of maximum 
degree $\le 3$ (to allow for 
each four-valent carbon atom retaining one valence for participation in a $\pi$ system). 
Eigenvectors and eigenvalues of chemical graphs correspond to the distributions of $\pi$ molecular orbitals and 
their energies within the
qualitative H{\"u}ckel model \cite{Streitwieser_1961}.
To qualify as a model of a carbon $\pi$-system, a nut graph must also be a chemical graph. 
In this model, the kernel eigenvector of a nut graph
corresponds to a non-bonding orbital, occupation of which by a single electron
leads to spin density, and hence radical reactivity distributed over all carbon centres 
\cite{Prop23}. 
Nut graphs also have unique status as \emph{strong omniconductors} \cite{PWF_CPL_2013,PWF_JCP_2014} of nullity one,
in the {H{\"u}ckel-based} version of the SSP (source-and-sink potential) 
model of ballistic molecular conduction \cite{Ern_JCP_2007a,Ern_JCP_2011b}.

For obvious reasons, we call a nut graph that satisfies the conditions
for a chemical graph a \emph{chemical nut graph}.
Given their chemical interest, a natural question is: for which combinations of 
order ($n$, number of vertices) and size ($m$, number of edges) do chemical nut graphs 
exist? It turns out to be possible to supply a complete answer to this question.

\section{Initial considerations}
A nut graph has various useful properties. 
These include the following: a nut graph is connected;
it is non-bipartite; it has no vertices of degree $1$ \cite{ScirihaGutman-NutExt}. 
We will refer to graphs that have no vertices of degree $1$ as \emph{leafless}.
Since the number of vertices of odd degree of a simple graph is even, 
a chemical nut graph is leafless, with
$v_2$ vertices of degree $2$, 
and $v_3$ vertices of degree $3$, where $v_3$ is even. 
Order and size of a chemical nut graph are related to these numbers by 
$n = v_2 + v_3$ and $m = v_2 + (3v_3)/2$.
This suggests a mapping of chemical nut graphs onto a grid parameterised by
even integers $v_3 \ge 0$ and  integers $v_2 \ge 0$.
 
Hence, our initial question can be rephrased as: for what combinations $(v_3, v_2)$ do chemical nut graphs exist? 
In other words: which pairs $(v_3,v_2)$ are \emph{realisable}? In this paper we give a complete answer to this question.
We also show that only one realisable pair, namely $(20,0)$ cannot be realised by a planar chemical nut graph.
(It is realised by chemical nut graphs of genus $1$.)
Useful background information exists, in the form of a database of nut graphs of low order 
produced by the computations described in \cite{CoolFowlGoed-2017} and listed on the accompanying website \cite{nutgen-site}.
Information 
available
on the website includes a dataset restricted to chemical nut graphs with $n \le 22$, 
which decides the existence question for small values of $v_3$ and $v_2$. 
Table~\ref{tbl:census} gives the results of our
interrogation of the dataset, listing cases of parameter pairs in the range 
$n \le 22$
where no chemical nut graph exists, and graph counts for the parameter pairs that are realisable by chemical nut 
graphs in the range.

\begin{table}[!ht]
\centering
\begin{small}
\begin{tabular}{|c|rrrrrrrrrrrr|}
\hline
\backslashbox{$v_2$}{$v_3$} & 0    & 2    & 4   & 6    & 8     & 10     & 12     & 14      & 16      & 18      & 20     & 22     \\  \hline
0         & $\emptyset$ & $\emptyset$ & \bf0   & \bf 0   & {\bf 0}     & {\bf 0}      & \bf9      & {\bf 0}       & {\bf 0}       & 5541    & \bf 5      & \bf 71     \\ 
1         & $\emptyset$ & $\emptyset$ & \bf0   & \bf0    & {\bf 0}     & {\bf 0}      & {\bf 0}      & \bf10      & \bf22      & \bf235     & 13602  & $-$    \\ 
2         & $\emptyset$ & \bf 0    & \bf0   &  \bf0    & {\bf 0}     & {\bf 0}      & \bf2      & {\bf 0}       & \bf37      & 3600    & \bf 30760  & $-$    \\ 
3         & 0    & \bf0    & \bf0   & \bf0    & \bf 7     & \bf9      & \bf71     & 5042    & 13474   & 168178  & $-$    & $-$    \\ 
4         & 0    & \bf0    & \bf0   & \bf0    & {\bf 0}     & \bf10     & 225    & \bf388     & 14022   & 480051  & $-$    & $-$    \\ 
5         & 0    & \bf  0    & \bf0   &\bf 0    & 7     & \bf82     & 596    & 16497   & 280798  & $-$     & $-$    & $-$    \\ 
6         & 0    & 0    & \bf0   & \bf0    & \bf4     & 127    & 1186   & 15801   & 545237  & $-$     & $-$    & $-$    \\ 
7         & 0    &\bf  1    & \bf0   & \bf 8    & 212   & 1368   & 23127  & 575614  & $-$     & $-$     & $-$    & $-$    \\ 
8         & 0    & 0    & \bf0   &\bf 5    & 22    & 620    & 12035  & 181009  & $-$     & $-$     & $-$    & $-$    \\ 
9         & 0    & 1    & \bf0   & 36   & 718   & 9603   & 211501 & $-$     & $-$     & $-$     & $-$    & $-$    \\ 
10        & 0    & 0    & \bf 2   & 13   & 176   & 5457   & 106013 & $-$     & $-$     & $-$     & $-$    & $-$    \\ 
11        & 0    & 3    & \bf0   & 189  & 4427  & 60792  & $-$    & $-$     & $-$     & $-$     & $-$    & $-$    \\ 
12        & 0    & 0    & 2   & 50   & 786   & 25535  & $-$    & $-$     & $-$     & $-$     & $-$    & $-$    \\ 
13        & 0    & 3    & \bf0   & 601  & 14153 & $-$    & $-$    & $-$     & $-$     & $-$     & $-$    & $-$    \\ 
14        & 0    & 0    & 11  & 118  & 3415  & $-$    & $-$    & $-$     & $-$     & $-$     & $-$    & $-$    \\ 
15        & 0    & 6    & 0   & 1881 & $-$   & $-$    & $-$    & $-$     & $-$     & $-$     & $-$    & $-$    \\ 
16        & 0    & 0    & 13  & 309  & $-$   & $-$    & $-$    & $-$     & $-$     & $-$     & $-$    & $-$    \\ 
17        & 0    & 6    & 0   & $-$  & $-$   & $-$    & $-$    & $-$     & $-$     & $-$     & $-$    & $-$    \\ 
18        & 0    & 0    & 38  & $-$  & $-$   & $-$    & $-$    & $-$     & $-$     & $-$     & $-$    & $-$    \\ 
19        & 0    & 10   & $-$ & $-$  & $-$   & $-$    & $-$    & $-$     & $-$     & $-$     & $-$    & $-$    \\ 
20        & 0    & 0    & $-$ & $-$  & $-$   & $-$    & $-$    & $-$     & $-$     & $-$     & $-$    & $-$    \\ 
21        & 0    & $-$  & $-$ & $-$  & $-$   & $-$    & $-$    & $-$     & $-$     & $-$     & $-$    & $-$    \\ 
22        & 0    & $-$  & $-$ & $-$  & $-$   & $-$    & $-$    & $-$     & $-$     & $-$     & $-$    & $-$    \\  \hline
\end{tabular}
\end{small}
\caption{Census of chemical nut graphs. 
The table shows the counts of chemical nut graphs for each of the allowed degree signatures 
$(v_3, v_2)$ 
accessible to graphs with $n \le 22$. Symbol - means that the case is outside the range of the dataset. Symbol $\emptyset$
in the table means that no graph with these parameters exists. 
Entries in the table were obtained by filtering the graphs from the nut graph database \cite{hog,nutgen-site,CoolFowlGoed-2017}.
Only the boldface entries are used in the proof of our main result; realisability in all other cases follows from the theory developed in Sections \ref{subsec:ext} to
\ref{subsec:nonrel}.}
\label{tbl:census}
\end{table}

As a chemical nut graph has $v_3$ even, we may consider it to be derived from a 
cubic graph on $v_3$ vertices with some edges arbitrarily subdivided to reach the total count of $v_2 + v_3$.  We will use this fact later
(e.g., in the proof of Theorem \ref{cor:realisx}).
Note that fast generation of chemical graphs was made possible by methods first introduced 
for cubic graphs \cite{hog,cubicpaper}.

\section{Main result}
The main existence result can be stated as follows.

\begin{theorem}
\label{thm:main}
A chemical nut graph with parameters $(v_3, v_2)$, $v_2 \geq 0, v_3 \geq 0$ and $v_3$ even, exists if and only if one of the following statements holds:
\begin{enumerate}[label=(\alph*)]
\item $v_3 = 2$  and  $v_2 = 7 + 2k $, $k \geq 0$\semicolon
\item $v_3 = 4$  and  $v_2 = 10 + 2k$, $k \geq 0$\semicolon
\item $v_3 = 6$  and  $v_2  \geq 7$\semicolon
\item $v_3 \geq 8$ and  
$$(v_3,v_2) \notin \{(8, 0), (8, 1), (8, 2), (8, 4), (10, 0), (10, 1), (10, 2), (12, 1), (14, 0), (14, 2), (16, 0)\}.$$
\end{enumerate}

Moreover, in each case where chemical nut graphs exist, a planar chemical nut graph may be found, except when
$v_3 = 20$ and  $v_2 = 0$, where one of the $5$ chemical nut graphs has genus $1$, and the other $4$ have genus $2$.
\end{theorem}

The key to proving this theorem,
and determining all realisable parameter pairs, is 
to use Table~\ref{tbl:census} as a source of seed graphs, and then to apply systematic constructions of larger nut graphs\cite{GPS,Jan}.
The proof of the main result has several
ideas. We identify each of them by a separate claim.

\subsection{Constructions for extending nut graphs}
\label{subsec:ext}

\subsubsection{The bridge construction}

The first construction that we introduce is \emph{the bridge construction} and is applicable only to
graphs with bridges, i.e.\ graphs with edges, whose removal disconnects the graph.
\begin{figure}[!h]
\centering
\begin{subfigure}{0.5\textwidth}
    \centering
    \begin{tikzpicture}
\definecolor{mygreen}{RGB}{205, 238, 231}
\tikzstyle{vertex}=[draw,circle,font=\scriptsize,minimum size=13pt,inner sep=1pt,fill=mygreen]
\tikzstyle{edge}=[draw,thick]
\draw[thick,fill=gray!20!white] (-1.2,0) ellipse (1.2cm and 1cm);
\draw[thick,fill=gray!20!white] (3.2,0) ellipse (1.2cm and 1cm);
\node at (-1.2, 0) {\huge $G_1$};
\node at (3.2, 0) {\huge $G_2$};
\node[vertex,fill=yellow,label={[yshift=0pt,xshift=5pt]90:$a$}] (u) at (0, 0) {$u$};
\node[vertex,fill=yellow,label={[xshift=-5pt,yshift=0pt]90:$b$}] (v) at (2, 0) {$v$};
\path[edge] (u) -- (v);
    \end{tikzpicture}
        \caption{$G$}
        \label{fig1a}
    \end{subfigure}%
\begin{subfigure}{0.5\textwidth}
    \centering
    \begin{tikzpicture}
\definecolor{mygreen}{RGB}{205, 238, 231}
\tikzstyle{vertex}=[draw,circle,font=\scriptsize,minimum size=13pt,inner sep=1pt,fill=mygreen]
\tikzstyle{edge}=[draw,thick]
\draw[thick,fill=gray!20!white] (-1.2,0) ellipse (1.2cm and 1cm);
\draw[thick,fill=gray!20!white] (4.2,0) ellipse (1.2cm and 1cm);
\node at (-1.2, 0) {\huge $-G_1$};
\node at (4.2, 0) {\huge $G_2$};
\node[vertex,fill=yellow,label={[yshift=0pt,xshift=5pt]90:$-a$}] (u) at (0, 0) {$u$};
\node[vertex,fill=yellow,label={[xshift=-5pt,yshift=0pt]90:$b$}] (v) at (3, 0) {$v$};
\node[vertex,fill=mygreen,label={[xshift=0pt,yshift=0pt]90:$-b$}] (x) at (1, 0) {$x$};
\node[vertex,fill=mygreen,label={[xshift=0pt,yshift=0pt]90:$a$}] (y) at (2, 0) {$y$};
\path[edge] (u) -- (x) -- (y) -- (v);
    \end{tikzpicture}
        \caption{$B(G, uv)$}
        \label{fig1b}
    \end{subfigure}%
\caption{The bridge construction. The graph $G$ consists of subgraphs $G_1$ and $G_2$ 
that are joined by an edge $uv$, i.e.\ $G_1$ and $G_2$ are connected components of $G - uv$.
$B(G, uv)$ is the enlargement of a nut graph by insertion of two vertices on a bridge and is also a nut graph.
Values $a$, $b$ and $-a$, $-b$, $a$ and $b$ 
are entries in the unique kernel eigenvectors of the graph
$G$ and its expansion, $B(G, uv)$, respectively.}
\label{fig:bridge}
\end{figure}
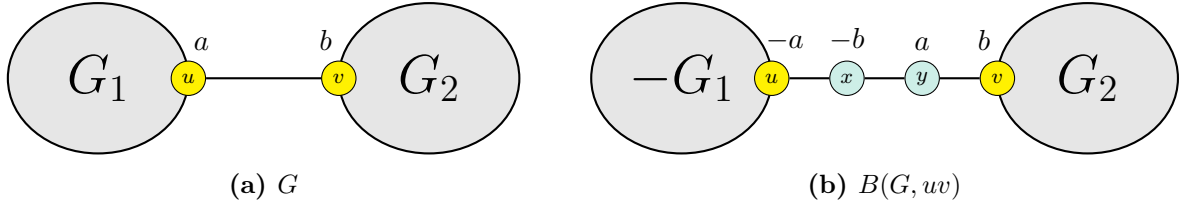
Let $uv$ be a bridge in $G$. By $B(G,uv)$ we denote the graph in which we insert two vertices on the bridge $uv$.

\begin{proposition}
\label{prop:4}
If we insert two vertices on a bridge $uv$ of $G$, the resulting graph $B(G,uv)$ is a nut graph if and only if $G$ is a nut graph.
\end{proposition}

\par\noindent
The forward direction of this result, together
with an alternative proof using linear algebra, 
can be found in Section 4.3 of \cite{ScirihaGutman-NutExt}.

\begin{proof}\ 

\noindent
$(\Rightarrow)$: The process of assigning entries of a new kernel eigenvector on the bridging path $u\,x\,y\,v$ in the graph $B(G, uv)$
from the graph $G$
is unique if we retain all kernel eigenvector entries 
in $G_2$.
(See Figure \ref{fig:bridge}.)

\noindent
$(\Leftarrow)$: If the graph $B(G, uv)$ is a nut graph and the vertices $u$ and $v$
have kernel eigenvector entries
assigned as shown, then the entries for vertices $x$ and $y$ follow.
The reverse operation can be carried out
by switching signs of all kernel eigenvector entries in either $G_1$ or $G_2$. 
Hence, the graph $G$ is a nut graph.
\end{proof}

\subsubsection{The subdivision construction}

\begin{proposition}
\label{prop:subdivision}
If we insert four vertices on an edge $uv$ of $G$, the resulting graph $S(G,uv)$ is a nut graph if and only if $G$ is a nut graph.
\end{proposition}
\par\noindent
The forward direction of this result, together
with an alternative proof using linear algebra, 
can be found as Lemma 4.1 in \cite{ScirihaGutman-NutExt}.

\begin{proof}\ 

\noindent $(\Rightarrow)$: The process of assigning entries of a new kernel eigenvector for the four inserted vertices $w\,x\,y\,z$ is
unique if we retain all eigenvector entries in $G$. (See Figure~\ref{fig:subdiv4}.)

\noindent $(\Leftarrow)$: If the enlarged graph $S(G, uv)$ is a nut graph,
then after removing the four additional vertices and joining $u$ directly to $v$,
retaining all other eigenvector entries, the reduced graph $G$ is also a nut graph.
\end{proof}

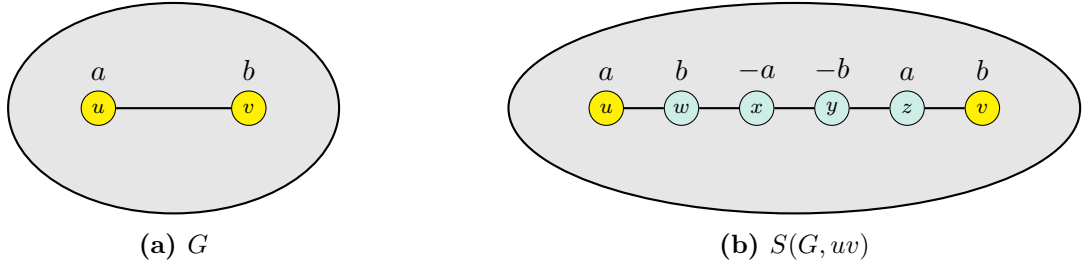
\begin{figure}[!htbp]
\centering
\begin{subfigure}{0.5\textwidth}
    \centering
    \begin{tikzpicture}
\definecolor{mygreen}{RGB}{205, 238, 231}
\tikzstyle{vertex}=[draw,circle,font=\scriptsize,minimum size=13pt,inner sep=1pt,fill=mygreen]
\tikzstyle{edge}=[draw,thick]
\draw[thick,fill=gray!20!white] (1,0) ellipse (2.2cm and 1.4cm);
\node[vertex,fill=yellow,label={[yshift=0pt,xshift=0pt]90:$a$}] (u) at (0, 0) {$u$};
\node[vertex,fill=yellow,label={[xshift=0pt,yshift=0pt]90:$b$}] (v) at (2, 0) {$v$};
\path[edge] (u) -- (v);
    \end{tikzpicture}
        \caption{$G$}
        \label{fig2a}
    \end{subfigure}%
\begin{subfigure}{0.5\textwidth}
    \centering
    \begin{tikzpicture}
\definecolor{mygreen}{RGB}{205, 238, 231}
\tikzstyle{vertex}=[draw,circle,font=\scriptsize,minimum size=13pt,inner sep=1pt,fill=mygreen]
\tikzstyle{edge}=[draw,thick]
\draw[thick,fill=gray!20!white] (2.5,0) ellipse (3.8cm and 1.4cm);
\node[vertex,fill=yellow,label={[yshift=0pt,xshift=0pt]90:$a$}] (u) at (0, 0) {$u$};
\node[vertex,fill=yellow,label={[xshift=0pt,yshift=0pt]90:$b$}] (v) at (5, 0) {$v$};
\node[vertex,fill=mygreen,label={[xshift=0pt,yshift=0pt]90:$b$}] (w) at (1, 0) {$w$};
\node[vertex,fill=mygreen,label={[xshift=0pt,yshift=0pt]90:$-a$}] (x) at (2, 0) {$x$};
\node[vertex,fill=mygreen,label={[xshift=0pt,yshift=0pt]90:$-b$}] (y) at (3, 0) {$y$};
\node[vertex,fill=mygreen,label={[xshift=0pt,yshift=0pt]90:$a$}] (z) at (4, 0) {$z$};
\path[edge] (u) -- (w) -- (x) -- (y) -- (z) -- (v);
    \end{tikzpicture}
        \caption{$S(G, uv)$}
        \label{fig2b}
    \end{subfigure}%

\caption{The subdivision construction. 
The graph
$S(G, uv)$ is obtained by
inserting vertices $w, x, y, z$ into the edge $uv$ of graph $G$.
Values $a$, $b$, $-a$, $-b$
are entries in the unique kernel eigenvectors of the graph
$G$ and its subdivision expansion, $S(G, uv)$.
}
\label{fig:subdiv4}
\end{figure}

\subsubsection{The \lq Fowler construction\rq }

The third construction is applicable to any graph $G$ and 
any of its vertices $v$ with degree $d$, $d > 1$.
Let $G$ be a graph and let $v$ be a vertex of degree $d$ 
in $G$. Let the neighbourhood of $v$ be
$N_G(v) = \{u_1,u_2,\ldots, u_d\}$.
We remove the $d$ edges incident on $v$,
add $2d$ vertices, and connect them to the rest of the graph as shown in Figure~\ref{fig:3}.
Let $F(G,v)$ denote the resulting graph, which has been 
called the \emph{Fowler construction} on $G$ \cite{GPS}.
The case where $d = 3$ was already described in \cite{Sc2008}.

Let the kernel eigenvector $\x$ be assigned to the vertices of $G$ in such a way that $\x(v) = x, \x(u_1) = x_1, \x(u_2) = x_2, \ldots , \x(u_d) = x_d$.
Then there is a unique way to carry over this kernel vector to $F(G,v)$. Moreover, we have the following theorem, 
proved by Gauci et al.\ in \cite{GPS}.

\begin{theorem}[\cite{GPS}]
\label{thm:6}
$G$ is a nut graph if and only if $F(G,v)$ is a nut graph.
\end{theorem}

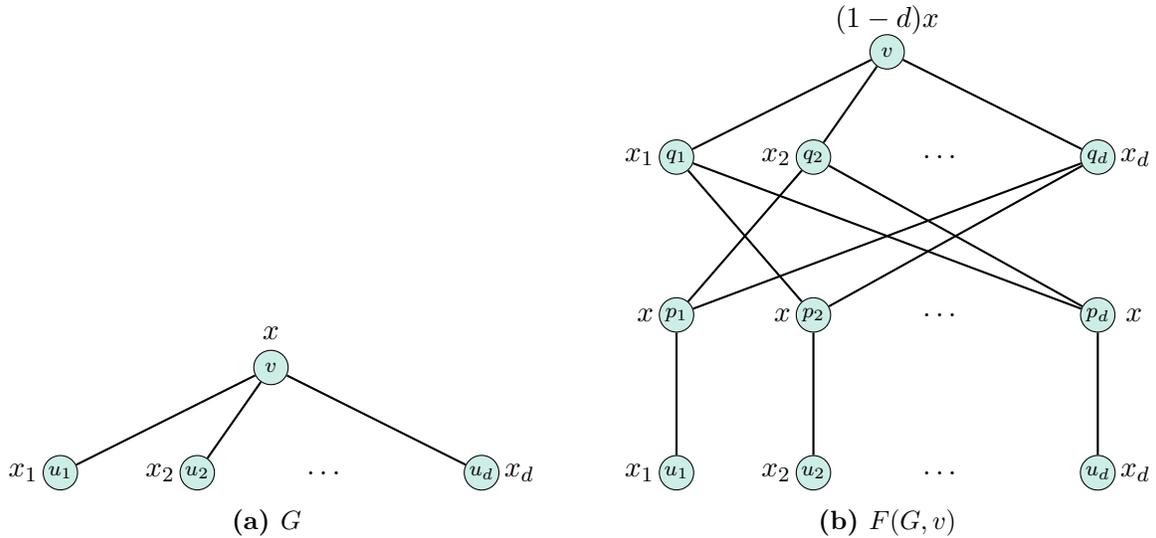
\begin{figure}[!h]
\centering
\subcaptionbox{$G$} 
[.5\linewidth]{
\begin{tikzpicture}[scale=1.4]
\definecolor{mygreen}{RGB}{205, 238, 231}
\tikzstyle{vertex}=[draw,circle,font=\scriptsize,minimum size=13pt,inner sep=1pt,fill=mygreen]
\tikzstyle{edge}=[draw,thick]
\node[vertex,fill=mygreen,label={[yshift=0pt,xshift=0pt]90:$x$}] (v) at (0, 0) {$v$};
\node[vertex,fill=mygreen,label={[xshift=2pt,yshift=0pt]180:$x_1$}] (u1) at (-2, -1) {$u_1$};
\node[vertex,fill=mygreen,label={[xshift=2pt,yshift=0pt]180:$x_2$}] (u2) at (-0.7, -1) {$u_2$};
\node[vertex,fill=mygreen,label={[xshift=-2pt,yshift=0pt]0:$x_d$}] (ud) at (2, -1) {$u_d$};
\node at (0.5, -1) {$\ldots$};
\path[edge] (v) -- (u1);
\path[edge] (v) -- (u2);
\path[edge] (v) -- (ud);
\end{tikzpicture}
}%
\subcaptionbox{$F(G, v)$}
[.5\linewidth]{\begin{tikzpicture}[scale=1.4]
\definecolor{mygreen}{RGB}{205, 238, 231}
\tikzstyle{vertex}=[draw,circle,font=\scriptsize,minimum size=13pt,inner sep=1pt,fill=mygreen]
\tikzstyle{edge}=[draw,thick]
\node[vertex,fill=mygreen,label={[yshift=-4pt,xshift=0pt]90:$(1-d)x$}] (v) at (0, 0) {$v$};
\node[vertex,fill=mygreen,label={[xshift=2pt,yshift=0pt]180:$x_1$}] (q1) at (-2, -1) {$q_1$};
\node[vertex,fill=mygreen,label={[xshift=2pt,yshift=0pt]180:$x_2$}] (q2) at (-0.7, -1) {$q_2$};
\node[vertex,fill=mygreen,label={[xshift=-2pt,yshift=0pt]0:$x_d$}] (qd) at (2, -1) {$q_d$};
\node at (0.5, -1) {$\ldots$};
\node[vertex,fill=mygreen,label={[xshift=2pt,yshift=0pt]180:$x$}] (p1) at (-2, -2.5) {$p_1$};
\node[vertex,fill=mygreen,label={[xshift=2pt,yshift=0pt]180:$x$}] (p2) at (-0.7, -2.5) {$p_2$};
\node[vertex,fill=mygreen,label={[xshift=0pt,yshift=0pt]0:$x$}] (pd) at (2, -2.5) {$p_d$};
\node at (0.5, -2.5) {$\ldots$};
\node[vertex,fill=mygreen,label={[xshift=2pt,yshift=0pt]180:$x_1$}] (u1) at (-2, -4) {$u_1$};
\node[vertex,fill=mygreen,label={[xshift=2pt,yshift=0pt]180:$x_2$}] (u2) at (-0.7, -4) {$u_2$};
\node[vertex,fill=mygreen,label={[xshift=-2pt,yshift=0pt]0:$x_d$}] (ud) at (2, -4) {$u_d$};
\node at (0.5, -4) {$\ldots$};
\path[edge] (v) -- (q1);
\path[edge] (v) -- (q2);
\path[edge] (v) -- (qd);
\path[edge] (p1) -- (q2); \path[edge] (p1) -- (qd);
\path[edge] (p2) -- (q1); \path[edge] (p2) -- (qd);
\path[edge] (pd) -- (q1); \path[edge] (pd) -- (q2);
\path[edge] (p1) -- (u1); 
\path[edge] (p2) -- (u2); 
\path[edge] (pd) -- (ud);
\end{tikzpicture}}
\caption{A construction for expansion of a nut graph $G$ about vertex $v$ of degree $d$, to give $F(G, v)$.
The labelling of vertices in $G$ and $F(G, v)$ is shown within the circles that represent vertices.
Shown beside each vertex is the corresponding entry of the unique kernel eigenvector of the respective graph.}
\label{fig:3}
\end{figure}

The construction 
$F(G,v)$
converts a nut graph that contains a vertex of given  
degree $d$ to a nut graph with $2d$ more vertices of that degree. It has been described in the literature 
for the case of general $d$, but here
the interesting cases are for degrees $2$ and $3$. 

If $v$ is a vertex of degree $2$ then it belongs to a path of length at least $2$. 
If we apply $F$ to this vertex, the length of the path increases by $4$.
If instead $v$ is a vertex of degree $3$, then the construction replaces it, and its neighbourhood, by the subgraph depicted in Figure~\ref{fig:4}.

\begin{figure}[!htb]
\centering
\begin{subfigure}{0.5\textwidth}
\centering
\begin{tikzpicture}
\definecolor{mygreen}{RGB}{205, 238, 231}
\tikzstyle{vertex}=[draw,circle,font=\scriptsize,minimum size=13pt,inner sep=1pt,fill=mygreen]
\tikzstyle{edge}=[draw,thick]
\node[vertex,fill=mygreen] (v) at (0, 0) {$v$};
\node[vertex,fill=mygreen] (u3) at (60:2.5) {$u_3$};
\node[vertex,fill=mygreen] (u2) at (-60:2.5) {$u_2$};
\node[vertex,fill=mygreen] (u1) at (180:2.5) {$u_1$};
\path[edge] (v) -- (u1);
\path[edge] (v) -- (u2);
\path[edge] (v) -- (u3);
\end{tikzpicture}
\caption{$G$}
\end{subfigure}%
\begin{subfigure}{0.5\textwidth}
\centering
\begin{tikzpicture}
\definecolor{mygreen}{RGB}{205, 238, 231}
\tikzstyle{vertex}=[draw,circle,font=\scriptsize,minimum size=13pt,inner sep=1pt,fill=mygreen]
\tikzstyle{edge}=[draw,thick]
\node[vertex,fill=mygreen] (v) at (0, 0) {$v$};
\node[vertex,fill=mygreen] (u3) at (60:2.5) {$u_3$};
\node[vertex,fill=mygreen] (u2) at (-60:2.5) {$u_2$};
\node[vertex,fill=mygreen] (u1) at (180:2.5) {$u_1$};
\node[vertex,fill=yellow] (y3) at (60:1.25) {};
\node[vertex,fill=yellow] (y2) at (-60:1.25) {};
\node[vertex,fill=yellow] (y1) at (180:1.25) {};
\node[vertex,fill=yellow] (t3) at (120:1.25) {};
\node[vertex,fill=yellow] (t2) at (0:1.25) {};
\node[vertex,fill=yellow] (t1) at (-120:1.25) {};
\path[edge] (y1) -- (u1);
\path[edge] (y2) -- (u2);
\path[edge] (y3) -- (u3);
\path[edge] (t1) -- (v);
\path[edge] (t2) -- (v);
\path[edge] (t3) -- (v);
\path[edge] (y1) -- (t1) -- (y2) -- (t2) -- (y3) -- (t3) -- (y1);
\end{tikzpicture}
\caption{$F(G, v)$}
\end{subfigure}
\caption{Construction $F(G,v)$ replaces cubic vertex $v$ and its neighbourhood $N_G(v)$ with the $10$-vertex subgraph shown on the right.}
\label{fig:4}
\end{figure}
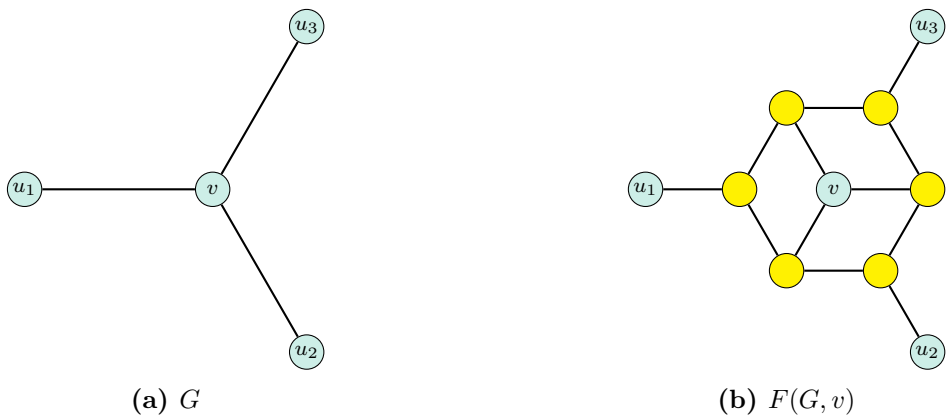
 
The process of extending the nut graph can be described as follows: take a nut graph $G$ that has a vertex of degree $d$. 
The non-trivial kernel eigenvector of the adjacency matrix of $G$ has entries $x$ on the vertex of interest 
and entries $u_i$ ($\sum_i u_i = 0$) on its neighbours. It is possible to construct a larger nut graph 
on $n+2d$ vertices with a full kernel eigenvector and the same nullity as $G$: 
the construction involves interleaving two layers of $d$ vertices internally connected as a cocktail-party graph. 
Untangling the cocktail-party portion of the graph shows that the larger graph inherits the genus of $G$, 
when the vertex at which expansion takes place is either
of degree $2$ or degree $3$. The following proposition explains this more precisely.
\begin{proposition}
If a chemical graph $G$ is embedded in some closed surface $\Sigma$, then any bridge construction $B(G,uv)$, 
subdivision construction $S(G, uv)$, 
 or Fowler construction $F(G,v)$ can be embedded
in the same surface $\Sigma$. In particular, $G$ is planar if and only if each of $B(G,uv)$, $S(G,uv)$ and $F(G,v)$ is planar.
Moreover, 
if any one of the graphs $G$,  $B(G,uv)$, $S(G,uv)$ or $F(G,v)$ is planar then all of them are planar.
\end{proposition}
\par\noindent
It is not hard to see, for instance, by Menger's Theorem, that $G$ has the same connectivity as $F(G,v)$. This implies the following proposition.

\begin{proposition}
Let $G$ be a chemical graph.
$G$ is polyhedral, i.e.\ planar and $3$-connected, if and only if for any vertex $v$ of $G$ the graph $F(G,v)$ is polyhedral.
\end{proposition}

In passing we note that if we generalise the Fowler construction in a natural way to multigraphs, then, for instance, the cube graph can be obtained as $F(G,v)$ where $G$ is the Theta graph.
A further observation relates to the logical connection between the constructed graph $F(G,v)$ and the subdivision
construction $S(G, uv)$.

\subsection{Realisability results}

The preceding section established tools that allow us to
draw the following three conclusions.

\begin{lemma}
\label{lem:9}
If $(v_3,v_2)$ is realisable by a graph with a bridge, then $(v_3,v_2+2)$ is also realisable (in the same surface) by such graph.
\end{lemma}

Lemma \ref{lem:9} follows from Proposition~\ref{prop:4}.

\begin{lemma}
\label{lem:10}
If $(v_3,v_2)$ is realisable, then $(v_3,v_2+4)$ is also realisable (in the same surface).
\end{lemma}

Lemma \ref{lem:10} can be established from Proposition~\ref{prop:subdivision}.

\begin{lemma}
\label{lem:11}
If $(v_3,v_2)$ is realisable, then $(v_3+6,v_2)$ is also realisable (in the same surface). Moreover if the former graph has a bridge, the latter also has a bridge.
\end{lemma}

Lemma \ref{lem:11} follows from Theorem~\ref{thm:6}.

\subsection{Non-realisability results}
\label{subsec:nonrel}

\begin{lemma}
\label{lem:12}
No chemical graph with $v_3 = 0$ is a nut graph.
\end{lemma}

\begin{proof}
The lemma claims that no cycle is a nut graph. By inspection of spectra, this is indeed the case. Namely, a cycle $C_n$ is singular if and only if
$n$ is divisible by four, in which case it has nullity $2$. Compare column 1 in Table~\ref{tbl:census}.
\end{proof}

\begin{lemma}
No chemical graph with
 $$(v_3,v_2) \in \{((8, 0), (8, 1), (8, 2), (8, 4), (10, 0), (10, 1), (10, 2), (12, 1), (14, 0), (14, 2), (16, 0))\}$$
 is a nut graph.
\label{lemma:cases}
\end{lemma}
\begin{proof}
The proof follows from computer determination of all chemical nut graphs on $n$ vertices, 
$n \leq 22$, as presented in Table~\ref{tbl:census}, stratified by $(v_3,v_2)$ parameters. The entries covered by this lemma are shown as boldface zeros
in the table.
\end{proof}
\par\noindent
Note that Table~\ref{tbl:census} also shows non-existence of some other pairs $(v_3,v_2)$ with $v_3+v_2 \leq 22$,
but these follow from Lamma \ref{lem:12} and Corollary \ref{cor:13} 

It turns out that only half of the pairs $(v_3, v_2)$ for $v_3 \in \{2, 4\}$  are only realisable (depending on the parity of $v_2$). To prove this, we need the following result.

\begin{theorem} 
\label{cor:realisx}
Let $G$ be a leafless chemical graph with $v_2 \geq \frac{9}{2}v_3 + 1$. Then the pair $(v_3,v_2)$ is realisable if and only if the pair $(v_3,v_2-4)$ is realisable. If the latter pair is not realisable
then none of the pairs $(v_3,v_2 + 4k)$, $k \in \{0,1,2,\ldots\}$, is realisable.
\end{theorem}

\begin{proof}
Since $G$ is a leafless chemical graph it can be viewed as a general subdivision of a cubic graph $H$ on $v_3$ vertices.
Hence, $H$ has $m = 3 v_3 / 2$ edges. And $G$ has its $v_2$ vertices of degree $2$ placed on these edges, forming corresponding paths.
As $v_2 \geq \frac{9}{2}v_3 + 1 = 3m+1$, by the Pigeonhole Principle, when forming $G$ from $H$, at least one of the edges of $H$ must be
replaced by a path containing at least $4$ degree-$2$ vertices. 
The last claim of this theorem follows by repeated application of Lemma~\ref{lem:10}.
\end{proof}

In the case $v_3 = 2$, Theorem~\ref{cor:realisx} applies to $v_2 \geq 10$; in the case $v_3 = 4$, it applies to $v_2 \geq 19$.

\begin{corollary} 
\label{cor:13}
Since pairs $(v_3,v_2) \in \{(2,6),(2,8),(4,15),(4,17)\}$ are not realisable, this means that no pair $(2,2k)$ and $(4,2k+1)$ is realisable, for $k \geq 0$. 
\end{corollary}

By using Corollary 13, we have generated two infinite series of zeros (non-realisable pairs), namely, for $v_3 = 2$ and $v_3 = 4$.
For the case $v_3 = 0$ we use Lemma 2 to obtain a third infinite series of zeros. The remaining zeros are finite in number
and are covered by Lemma 3 based on the database of chemical graphs in House of Graphs. 
The final step in the proof of Theorem 1 is to establish that outside the triangular region of
parameter space covered by the database, every pair that does
not belong to one of the infinite series of zeros (unrealisable cases), is realisable.
To do this we use the notion of seed graphs.

\subsection{Seed graphs}

A \emph{seed graph} is a chemical nut graph that cannot be generated from any smaller chemical nut graph by  
some combination of the three constructions.
We will show that all realisable pairs can be generated from a small set of seed graphs by repeated application of the constructions.

\begin{lemma}
Chemical nut graphs with the following parameters exist:
\begin{enumerate}[label=(\alph*)]
\item 
Planar, with a bridge (B):
$$(v_3,v_2) \in \{(2,7),(4,10),(6,7),(6,8),(8,5),(8,6),(10,4),(10,5),(21,3),(14,4),(20,2)\} $$
\item
Planar $2$-connected (P): 
$$(v_3,v_2) \in \{ (10,3),(12,2),(14,1),(16,1),(16,2),(18,1),(22,0) \}$$
\item
Polyhedral (i.e.\ planar $3$-connected) ($\mathit{\Pi}$): 
$$(v_3,v_2) \in \{ (12,0),(26,0), (28, 0) \}$$
\item
Non-planar, genus one (N): $$(v_3,v_2)  \in \{  (20,0) \}$$
\end{enumerate}
\end{lemma}

\begin{proof}
By inspection of adjacency data presented in the appendices
for each seed chemical nut graph. 
\end{proof}

Note that other chemical nut graphs with these parameters
may exist and that other choices of seed graphs are possible.

\subsection{Proof of Theorem~{\ref{thm:main}}}

\begin{proof}[Proof]
To prove the realisability result we construct Table 2 according to the following rules.
We start by indicating the parameters of the $22$
seed graphs listed in Lemma 14 using boldface B, P, $\Pi$, N as appropriate.
A bold X is used to indicate a non-realisable entry derived from Table 1.
Other entries are filled in according to three rules based on
lemmas:
\begin{enumerate}[label=(\roman*)]
\item 
 (Lemma~{\ref{lem:9}}) if we have B in the table (meaning a planar chemical nut graph with a bridge) then we can insert 
B two entries below;
\item (Lemma~{\ref{lem:10}}) if we have P or $\Pi$ in the table (meaning a planar chemical nut graph without a bridge) then we 
can insert P four entries
below; 
\item  (Lemma~\ref{lem:11}) if we have B, P or $\Pi$ in a row of the table then three steps to the right we have B, P or $\Pi$, respectively. 
 \end{enumerate}
We proceed initially by columns. 
Column $v_3 = 0$ is entirely non-realisable (see Lemma \ref{lem:12}).
The first boldface entry in each of columns $v_3 = 2$ 
and $v_3 = 4$  gives a series of realisable cases.
All other entries in these columns
are non-realisable: this follows from Corollary \ref{cor:13}. 
In column $v_3 = 6$, 
consecutive boldface B entries prove realisability for all $v_2 \geq 7$. Likewise, in column $v_3 = 8$, entry
$(3, 8)$ in combination with $(6, 8)$ proves realisability for all $v_2 \geq 5$ (by a planar graph containing a bridge). 
Similarly, in column $v_3 = 10$ 
the
entries $(4, 10)$ and $(5, 10)$ prove realisability for all $v_2 \geq 4$. Having established the rectangle
of six
cases outlined in the table (with corners 
$(7, 6)$, $(7, 10)$, $(8, 6)$, and $(8, 10)$), we fill the whole quadrant to the right and below.
Rows $v_2 = 7$ and $v_2 = 8$ are generated by the Fowler construction and all entries below by repetition of the bridge construction.
We are left with rows for $v_2 < 7$. For each of them we start with a seed graph and then repeat the Fowler construction.
Non-realisability of remaining entries follows from Table 1.
By using $(0, 26)$ as a seed graph we were able to show the existence of planar graphs for all entries to the right.
\end{proof}

\vspace{\baselineskip}

\newcommand{\mbg}{} 
\newcommand{\ybg}{} 
\newcommand{\bbg}{} 

\begin{table}[!h]
\centering
\begin{small}
\begin{tabular}{|c|rrr:rrrrrrrrrrrrr|}
 \hline
\backslashbox{$v_2$}{$v_3$}  & 0    & 2    & \multicolumn{1}{r}{4 }    & 6    & 8     & 10     & 12     & 14      & 16      & 18      & 20     & 22  & 24 & 26 & 28 & $\cdots$   \\  \hline
$0$ & $\emptyset$ & $\emptyset$ & \multicolumn{1}{r}{${\bf X}$}  & ${\bf X}$ & ${\bf X}$ & ${\bf X}$ & ${\bf \Pi}$ & ${\bf X}$ & ${\bf X}$ & $F\mathit{\Pi}$ & ${\bf N}$ & $\ybg {\bf P}$ & $\ybg F\mathit{\Pi}$ & $\ybg {\bf \Pi}$ & ${\bf \Pi}$ & $\Rightarrow$ \\
$1$ & $\emptyset$ & $\emptyset$ & \multicolumn{1}{r}{${\bf X}$} & ${\bf X}$ & ${\bf X}$ & ${\bf X}$ & ${\bf X}$ & $\ybg {\bf P}$ & $\ybg {\bf P}$ & $\ybg  {\bf P}$ & $FP$ & $FP$ & $FP$ & $FP$ & $FP$ & $\Rightarrow$ \\
$2$ & $\emptyset$ & ${\bf X}$ & \multicolumn{1}{r}{${\bf X}$} & ${\bf X}$ & ${\bf X}$ & ${\bf X}$ & ${\bf P}$ & ${\bf X}$ & $\ybg {\bf P}$ & $\ybg  FP$ & $\ybg  {\bf B}$ & $FP$ & $FP$ & $FB$ & $FP$ & $\Rightarrow$ \\
$3$ & $X$ & ${\bf X}$ & \multicolumn{1}{r}{${\bf X}$} & ${\bf X}$ & $\ybg {\bf B}$ & $\ybg {\bf P}$ & $\ybg {\bf B}$ & $FB$ & $FP$ & $FB$ & $FB$ & $FP$ & $FB$ & $FB$ & $FP$ & $\Rightarrow$ \\
$4$ & $X$ & ${\bf X}$ & \multicolumn{1}{r}{${\bf X}$} & ${\bf X}$ & ${\bf X}$ & $\ybg {\bf B}$ & $\ybg  P$ & $\ybg {\bf B}$ & $FB$ & $P$ & $FB$ & $FB$ & $FP$ & $FB$ & $FB$ & $\Rightarrow$ \\
$5$ & $X$ & ${\bf X}$ & \multicolumn{1}{r}{${\bf X}$} & ${\bf X}$ & $\ybg B$ & $\ybg {\bf B}$ & $\ybg B$ & $FB$ & $FB$ & $FB$ & $FB$ & $FB$ & $FB$ & $FB$ & $FB$ & $\Rightarrow$ \\
$6$ & $X$ & $\bbg X$ & \multicolumn{1}{r}{${\bf X}$} & ${\bf X}$ & $\ybg {\bf B}$ & $\ybg B$ & $\ybg P$ & $FB$ & $FB$ & $FP$ & $FB$ & $FB$ & $FP$ & $FB$ & $FB$ & $\Rightarrow$ \\
\cline{5-7} \cdashline{8-16}
$7$ & $X$ & $\bbg {\bf B}$ & ${\bf X}$ & \multicolumn{1}{|r}{$\mbg {\bf B}$ } & $\mbg B$ & \multicolumn{1}{r|}{$\mbg B$ } & $FB$ & $FB$ & $FB$ & $FB$ & $FB$ & $FB$ & $FB$ & $FB$ & $FB$ & $\Rightarrow$ \\
$8$ & $X$ & $\bbg X$ & ${\bf X}$ & \multicolumn{1}{|r}{$\mbg {\bf B} $} & $\mbg B$ &\multicolumn{1}{r|}{$\mbg B$ } & $FB$ & $FB$ & $FB$ & $FB$ & $FB$ & $FB$ & $FB$ & $FB$ & $FB$ & $\Rightarrow$ \\
\cline{5-7} 
$9$ & $X$ & $\bbg B$ & ${\bf X}$ & $B$ & $B$ & $B$ & $B$ & $B$ & $B$ & $B$ & $B$ & $B$ & $B$ & $B$ & $B$ & $\Rightarrow$ \\
$10$ & $X$ & $X$ & ${\bf B}$ & $B$ & $B$ & $B$ & $B$ & $B$ & $B$ & $B$ & $B$ & $B$ & $B$ & $B$ & $B$ & $\Rightarrow$ \\
$11$ & $X$ & $B$ & ${\bf X}$ & $B$ & $B$ & $B$ & $B$ & $B$ & $B$ & $B$ & $B$ & $B$ & $B$ & $B$ & $B$ & $\Rightarrow$ \\
$12$ & $X$ & $X$ & $B$ & $B$ & $B$ & $B$ & $B$ & $B$ & $B$ & $B$ & $B$ & $B$ & $B$ & $B$ & $B$ & $\Rightarrow$ \\
$13$ & $X$ & $B$ & ${\bf X}$ & $B$ & $B$ & $B$ & $B$ & $B$ & $B$ & $B$ & $B$ & $B$ & $B$ & $B$ & $B$ & $\Rightarrow$ \\
$14$ & $X$ & $X$ & $B$ & $B$ & $B$ & $B$ & $B$ & $B$ & $B$ & $B$ & $B$ & $B$ & $B$ & $B$ & $B$ & $\Rightarrow$ \\
$15$ & $X$ & $B$ & $\bbg X$ & $B$ & $B$ & $B$ & $B$ & $B$ & $B$ & $B$ & $B$ & $B$ & $B$ & $B$ & $B$ & $\Rightarrow$ \\
$16$ & $X$ & $X$ & $\bbg B$ & $B$ & $B$ & $B$ & $B$ & $B$ & $B$ & $B$ & $B$ & $B$ & $B$ & $B$ & $B$ & $\Rightarrow$ \\
$17$ & $X$ & $B$ & $\bbg X$ & $B$ & $B$ & $B$ & $B$ & $B$ & $B$ & $B$ & $B$ & $B$ & $B$ & $B$ & $B$ & $\Rightarrow$ \\
$18$ & $X$ & $X$ & $\bbg B$ & $B$ & $B$ & $B$ & $B$ & $B$ & $B$ & $B$ & $B$ & $B$ & $B$ & $B$ & $B$ & $\Rightarrow$ \\[-3pt]
$\vdots$ & $\Downarrow$ & $\Downarrow$ &  \multicolumn{1}{r}{$\Downarrow$}  & $\Downarrow$ & $\Downarrow$ & $\Downarrow$ & $\Downarrow$ & $\Downarrow$ & $\Downarrow$ & $\Downarrow$ & $\Downarrow$ & $\Downarrow$ & $\Downarrow$ & $\Downarrow$ & $\Downarrow$ & \\
\hline
\end{tabular}
\end{small}
\caption{Parameter pairs $(v_3,v_2)$ for which chemical nut graphs exist. Notation: $\emptyset$ - no graph exists; $X$ - no nut graph exists; $B$ - a planar graph with a bridge exists ($FB$ - obtained by Fowler construction); $P$ - a $2$-connected planar graph exists ($FP$ - obtained by Fowler construction); $\mathit{\Pi}$ - a polyhedral graph exists ($F\mathit{\Pi}$ - obtained by Fowler construction); $N$ - only non-planar graphs exist; $\Rightarrow$ - the pattern extends indefinitely to the right; $\Downarrow$ - the pattern extends indefinitely below. Boldface symbols were determined from the graph database: {\bf B, P, $\bf \Pi$, N} represent seed graphs, while {\bf X} represents non-existence that follows from Table 1.
Dashed lines indicate the quadrant that can be completely filled with bridged cases, once the boxed rectangle is established.
}
\label{tbl:map}
\end{table}

\section{Polyhedral and toroidal chemical nut graphs}
Before we conclude this paper we want to address a special class of chemical nut graphs, namely the polyhedral chemical nut graphs, i.e.\ 
$3$-connected planar chemical nut graphs. In practice this means that only cubic graphs, $v_2 = 0$, are eligible. The smallest chemical nut graphs of this type were found
in \cite{Prop23} and the list was extended in
\cite{hog,nutgen-site,CoolFowlGoed-2017,GPS}.
There are two cubic polyhedral nut graphs on $12$ vertices,
illustrated in \cite{Prop23}. One of these (Figure 17(b) \cite{Prop23}) is the Frucht graph \cite{GPS}; the other
(Figure 17(a) \cite{Prop23})
can be realised with $C_2$ point-group symmetry.
Since the result of applying the Fowler construction to a polyhedral nut graph is a polyhedral nut graph, these seed graphs give rise to an infinite series of parameters
$v_3 = 12,18,24, \ldots $ that admit polyhedral nuts.  
In \cite{GPS}, a cubic nut graph with $v_3 = 26$
is identified. In fact, this is one of four.  
By the same argument
we obtain another infinite series $v_3 = 26,32,38, \ldots $.  From Table 2 one can see that no polyhedral nut graph exists for $v_3 = 14,16,,20$. 
There is no cubic polyhedral nut graph with $v_3 =22$ (\cite{Prop23}, 
confirmed in \cite{hog,nutgen-site,CoolFowlGoed-2017,GPS}).  
There are 316 chemical polyhedral nut graphs on $v_3 = 28$ vertices \cite{hog,nutgen-site,CoolFowlGoed-2017}. This gives rise to the final infinite series: $v_3 = 28,34,40, \ldots$. The adjacency
data of the three polyhedral
seeds are given in the Appendix $\Pi$.
These observations completely characterise admissible parameters and hence prove the following proposition.

\begin{proposition}
A polyhedral chemical nut graph with parameters $(v_3,v_2)$ exists if and only if $v_2 = 0$ and $v_3$ is even with $v_3 = 12$, $v_3 = 18$ or $v_3 \geq 24$.
\end{proposition}

A subset of the cubic polyhedra of special interest
for the chemistry of carbon cages
is that of the fullerenes. A \emph{fullerene} has a cubic polyhedral molecular graph in which exactly  $12$ faces are of size $5$ 
and any other faces are of size $6$. Nut fullerenes seem
relatively rare \cite{Prop23}, and have been enumerated up to order $240$ 
\cite{hog,nutgen-site,CoolFowlGoed-2017}. 
The smallest nut fullerene has $36$ vertices and a construction given in \cite{Prop23} (Figure 14) shows that the unique 
C$_{36}$ nut fullerene can be extended by adding belts of six hexagons at a time to make larger cylindrical fullerenes that are also nut graphs. This rationalises the presence of the integers $n = 24+12k$ ($k > 0$) in the published list of orders of fullerene nut graphs \cite{hog,nutgen-site,CoolFowlGoed-2017}. The members of this sub-series all have the same six-pentagon cluster in each cylinder cap, and all are \emph{uniform} \cite{Prop23} in the sense that every vertex is surrounded in the kernel eigenvector by the same triple of entries $\{2, -1, -1\}$.

All nut fullerenes known so far have multiple pentagon adjacencies, which militates against
their stability as neutral all-carbon molecules.

Incidentally, as there is, in addition to the example on $20$ vertices, a chemical nut graph of genus $1$ on $22$ and $24$ vertices (see Appendix N),
we have an infinite series of toroidal $3$-regular chemical nut graphs. More precisely, there is a toroidal $3$-regular nut graph on $n$ vertices
if and only if $n$ is 
even and $n \geq 20$.

\section{Conclusion}

The parameter space occupied by chemical nut graphs
has been characterised in terms of 
allowed vertex signatures $(v_3, v_2)$,
according to Theorem {\ref{thm:main}}.
This result
has implications for the 
theory of molecular conduction, in that each 
chemical nut graph
corresponds to the carbon skeleton of a $\pi$-conjugated 
molecule with a single non-bonding molecular orbital
and strong omni-conducting 
behaviour \cite{PWF_CPL_2013,PWF_JCP_2014}.

The characterisation of chemical nut graphs by vertex degrees
also implies restrictions on their orders and sizes,
$n$ and $m$.
As a chemical nut graph has no vertices of degree $1$ and maximum degree $\Delta \le 3$, the Theorem \ref{thm:main} implies
a restricted range of Betti numbers, $m-n+1$, for these graphs.
For a chemical nut graph, $v_1 = 0$, $v_2 \ge 0$
and $v_3 > 0$ is even. Hence, $n = v_2 + v_3$, 
$m = v_2 + (3/2) v_3$, and the Betti numbers span the range
\begin{equation}
1 < \frac{v_3}{2} +1 =  m-n+1 \leq \left\lceil{\frac{n + 1}{2}}\right\rceil
\label{eq:range}
\end{equation}
Transforming Table 2 from $(v_3, v_2)$ to $(n, m)$
space, shows that for $n = 15$ and $n \ge 17$, there is 
just one missing value in the range \eqref{eq:range}:
for odd $n \ge 15$ no chemical nut graph has $m-n+1 =3$, and for even $n \ge 18$ no chemical nut graph has $m-n+1 =2$.  
For all other realisable values of $n$ (more precisely, $9 \leq n \leq 14$ and $n = 16$), there are at least two missing values in the range \eqref{eq:range}.

Although we have settled the question of the existence of chemical nut graphs, there several intriguing related topics that we can
address in future work. One of them is the question of the existence of signed chemical nut graphs. For signed graphs, see for instance recent work \cite{BeCiKoWa2019} and~\cite{GhHaMaMa2020}. 
  
\section*{Acknowledgements}
The work of Toma\v{z} Pisanski is supported in part by the Slovenian Research Agency (research program P1-0294 
and research projects J1-2481, N1-0032, J1-9187 and J1-1690), 
and in part by H2020 Teaming InnoRenew CoE.
The work of Nino Bašić is supported in part by the Slovenian Research Agency (research program P1-0294
and research projects J1-2481, J1-9187, J1-1691 and N1-0140).

\bibliographystyle{amcjoucc}
\bibliography{references}

\pagebreak
\section*{Appendix S: Seed graphs} 

For each seed graph $G$, the entry lists $v_3$, $v_2$, $n$, $m$
followed by $m$ edges $E(G)$ and the kernel eigenvector $\bf x$ as
a list of integer entries.

\vspace{\baselineskip}
\noindent
$\boldsymbol{v_3 = 2, v_2 = 7, n=  9, m = 10}$

\noindent
$
E(G) = \{ (0, 1), (1, 2), (2, 3), (3, 4), (4, 5), (5, 6), (6, 7), (7, 8), (0, 2), (4, 8) \}
$

\noindent
$
{\bf x} = \begin{bmatrix} 1 & 1 & -1 & -2 & 1 & 1 & -1 & -1 & 1 \end{bmatrix} 
$

\vspace{\baselineskip}
\noindent
$\boldsymbol{v_3 = 4, v_2 = 10, n=  14, m = 16}$

\noindent
$
E(G) = \{ (0, 6), (0, 10), (1, 7), (1, 10), (2, 8), (2, 11), (3, 9), (3, 13), (4, 10), (4, 11), (5, 12), (5, 13), \\(6, 12),  (7, 12), (8, 11), (9, 13) \}
$

\noindent
$
{\bf x} = \begin{bmatrix} 1 & 1 & 1 & -1 & -2 & 2 & -1 & -1 & 1 & -1 & 1 & -1 & -1 & 1 \end{bmatrix} 
$

\vspace{\baselineskip}
\noindent
$\boldsymbol{v_3 = 6, v_2 = 7, n=  13, m = 16}$

\noindent
$
E(G) = \{ (0, 6), (0, 10), (1, 7), (1, 8), (1, 10), (2, 7), (2, 9), (2, 12), (3, 8), (3, 11), (4, 9), (4, 11), (5, 11),\\ (5, 12), (6, 10), (9, 12) \}
$

\noindent
$
{\bf x} = \begin{bmatrix} 1 & -2 & 2 & 2 & -1 & -1 & 1 & 2 & -1 & -1 & -1 & 1 & -1 \end{bmatrix} 
$

\vspace{\baselineskip}
\noindent
$\boldsymbol{v_3 = 6, v_2 = 8, n=  14, m = 17}$

\noindent
$
E(G) = \{ (0, 6), (0, 8), (0, 13), (1, 7), (1, 12), (2, 9), (2, 10), (3, 9), (3, 11), (3, 13), (4, 10), (4, 11), (5, 11),\\  (5, 12), (5, 13), (6, 8), (7, 12) \}
$

\noindent
$
{\bf x} = \begin{bmatrix} 1 & 1 & -1 & 1 & 1 & -2 & -1 & 1 & -1 & -1 & 1 & -1 & -1 & 2 \end{bmatrix} 
$

\vspace{\baselineskip}
\noindent
$\boldsymbol{v_3 = 8, v_2 = 3, n=  11, m = 15}$

\noindent
$
E(G) = \{ (0, 5), (0, 9), (1, 6), (1, 7), (1, 10), (2, 6), (2, 8), (2, 10), (3, 7), (3, 8), (3, 10), (4, 8), (4, 9), \\(5, 9), (6, 7) \}
$

\noindent
$
{\bf x} = \begin{bmatrix} 1 & -2 & 1 & 1 & -2 & 1 & 1 & 1 & 1 & -1 & -2 \end{bmatrix} 
$

\vspace{\baselineskip}
\noindent
$\boldsymbol{v_3 = 8, v_2 = 6, n=  14, m = 18}$

\noindent
$
E(G) = \{ (0, 5), (0, 9), (0, 13), (1, 6), (1, 10), (1, 11), (2, 7), (2, 10), (3, 8), (3, 12), (4, 9), (4, 11), (5, 11), \\(6, 10), (7, 12), (7, 13), (8, 12), (9, 13) \}
$

\noindent
$
{\bf x} = \begin{bmatrix} 1 & -2 & 3 & 1 & 1 & 1 & -1 & -2 & 1 & 1 & 2 & -1 & -1 & -2 \end{bmatrix} 
$

\vspace{\baselineskip}
\noindent
$\boldsymbol{v_3 = 10, v_2 = 3, n=  13, m = 18}$

\noindent
$
E(G) = \{ (0, 5), (0, 8), (1, 6), (1, 9), (2, 7), (2, 11), (2, 12), (3, 8), (3, 9), (3, 10), (4, 8), (4, 9), (4, 12),\\ (5, 10), (5, 11), (6, 10), (6, 12), (7, 11) \}
$

\noindent
$
{\bf x} = \begin{bmatrix} 2 & 2 & 1 & 1 & -3 & -3 & 2 & 2 & 3 & -2 & -1 & -1 & -1 \end{bmatrix} 
$

\vspace{\baselineskip}
\noindent
$\boldsymbol{v_3 = 10, v_2 = 4, n=  14, m = 19}$

\noindent
$
E(G) = \{ (0, 5), (0, 9), (0, 10), (1, 6), (1, 9), (1, 11), (2, 7), (2, 11), (3, 8), (3, 12), (3, 13), (4, 9), (4, 12), \\(4, 13), (5, 10), (5, 13), (6, 10), (7, 11), (8, 12) \}
$

\noindent
$
{\bf x} = \begin{bmatrix} 1 & 2 & -1 & 2 & -3 & 1 & -2 & -1 & 1 & 1 & -2 & 1 & -2 & 1 \end{bmatrix} 
$

\vspace{\baselineskip}
\noindent
$\boldsymbol{v_3 = 10, v_2 = 5, n=  15, m = 20}$

\noindent
$
E(G) = \{ (0, 6), (0, 8), (0, 11), (1, 7), (1, 9), (1, 14), (2, 8), (2, 10), (2, 12), (3, 9), (3, 13), (3, 14), (4, 10),\\ (4, 11), (5, 12), (5, 13), (6, 12), (7, 13), (8, 11), (9, 14) \}
$

\noindent
$
{\bf x} = \begin{bmatrix} 2 & -2 & 1 & 3 & -1 & -5 & 4 & 2 & -1 & -1 & 3 & -3 & -2 & 2 & -1 \end{bmatrix} 
$

\vspace{\baselineskip}
\noindent
$\boldsymbol{v_3 = 12, v_2 = 0, n=  12, m = 18}$

\noindent
$
E(G) = \{ (0, 4), (0, 7), (0, 8), (1, 5), (1, 7), (1, 9), (2, 6), (2, 9), (2, 11), (3, 7), (3, 10), (3, 11), (4, 8), \\ (4, 10), (5, 8), (5, 9), (6, 10), (6, 11) \}
$

\noindent
$
{\bf x} = \begin{bmatrix} 1 & 1 & 1 & -2 & 1 & -2 & 1 & 1 & -2 & 1 & 1 & -2 \end{bmatrix} 
$

\vspace{\baselineskip}
\noindent
$\boldsymbol{v_3 = 12, v_2 = 2, n=  14, m = 20}$

\noindent
$
E(G) = \{ (0, 5), (0, 9), (0, 10), (1, 6), (1, 9), (1, 11), (2, 7), (2, 12), (2, 13), (3, 8), (3, 10), (4, 8), (4, 11), \\(4, 13), (5, 9), (5, 12), (6, 10), (6, 11), (7, 12), (7, 13) \}
$

\noindent
$
{\bf x} = \begin{bmatrix} 1 & -3 & -1 & -2 & 2 & 2 & 1 & -1 & -1 & -3 & 1 & 2 & 2 & -1 \end{bmatrix} 
$

\vspace{\baselineskip}
\noindent
$\boldsymbol{v_3 = 12, v_2 = 3, n=  15, m = 21}$

\noindent
$
E(G) = \{ (0, 5), (0, 9), (0, 12), (1, 6), (1, 10), (1, 11), (2, 7), (2, 10), (2, 14), (3, 8), (3, 12), (3, 14), (4, 9),\\ (4, 13), (4, 14), (5, 12), (5, 13), (6, 10), (6, 11), (7, 11), (8, 13) \}
$

\noindent
$
{\bf x} = \begin{bmatrix} 3 & -1 & 2 & 1 & -3 & -4 & -1 & 2 & 7 & 6 & 3 & -2 & -2 & -1 & -5 \end{bmatrix} 
$

\vspace{\baselineskip}
\noindent
$\boldsymbol{v_3 = 14, v_2 = 1, n=  15, m = 22}$

\noindent
$
E(G) = \{ (0, 5), (0, 8), (0, 10), (1, 6), (1, 11), (1, 12), (2, 7), (2, 13), (2, 14), (3, 9), (3, 10), (3, 12), (4, 9), \\(4, 11), (4, 14), (5, 8), (5, 10), (6, 12), (6, 14), (7, 13), (8, 13), (9, 11) \}
$

\noindent
$
{\bf x} = \begin{bmatrix} 1 & 3 & 2 & -2 & -1 & 1 & -1 & 3 & -5 & -2 & 4 & 3 & -2 & -2 & -1 \end{bmatrix} 
$

\vspace{\baselineskip}
\noindent
$\boldsymbol{v_3 = 14, v_2 = 4, n=  18, m = 25}$

\noindent
$
E(G) = \{ (0, 8), (0, 16), (1, 9), (1, 10), (1, 17), (2, 9), (2, 11), (2, 17), (3, 10), (3, 12), (3, 13), (4, 11), \\(4, 15), (4, 16), (5, 12), (5, 13), (5, 14), (6, 12), (6, 14), (6, 15), (7, 13), (7, 14), (7, 15), (8, 16), (9, 17) \}
$

\noindent
$
{\bf x} = \begin{bmatrix} 1 & -1 & 2 & 1 & -2 & -2 & 1 & 1 & 1 & -1 & 2 & 2 & -1 & -1 & 2 & -1 & -1 & -1 \end{bmatrix} 
$

\vspace{\baselineskip}
\noindent
$\boldsymbol{v_3 = 16, v_2 = 1, n=  17, m = 25}$

\noindent
$
E(G) = \{ (0, 6), (0, 9), (0, 11), (1, 7), (1, 12), (1, 16), (2, 8), (2, 14), (3, 10), (3, 13), (3, 16), (4, 10),\\ (4, 14), (4, 15), (5, 11), (5, 12), (5, 13), (6, 9), (6, 11), (7, 13), (7, 14), (8, 15), (8, 16), (9, 12), (10, 15) \}
$

\noindent
$
{\bf x} = \begin{bmatrix} 1 & -1 & -3 & -4 & -3 & -2 & 1 & 6 & 5 & 3 & -2 & -4 & -2 & 6 & -5 & 7 & -4 \end{bmatrix} 
$

\vspace{\baselineskip}
\noindent
$\boldsymbol{v_3 = 16, v_2 = 2, n=  18, m = 26}$

\noindent
$
E(G) = \{ (0, 7), (0, 9), (0, 15), (1, 8), (1, 15), (1, 17), (2, 9), (2, 13), (2, 14), (3, 10), (3, 12), (3, 13), \\(4, 10), (4, 14), (5, 11), (5, 12), (5, 16), (6, 11), (6, 13), (6, 17), (7, 14), (7, 15), (8, 16), (9, 17), (10, 12),\\ (11, 16) \}
$

\noindent
$
{\bf x} = \begin{bmatrix} 12 & -9 & -6 & 3 & 9 & -12 & 3 & -3 & 9 & 6 & 9 & 3 & -12 & 3 & -9 & -3 & 9 & -6 \end{bmatrix} 
$

\vspace{\baselineskip}
\noindent
$\boldsymbol{v_3 = 18, v_2 = 1, n=  19, m = 28}$

\noindent
$
E(G) = \{ (0, 7), (0, 9), (0, 13), (1, 8), (1, 14), (1, 18), (2, 10), (2, 11), (2, 14), (3, 10), (3, 16), (3, 17),\\ (4, 11), (4, 15), (4, 16), (5, 12), (5, 13), (5, 15), (6, 12), (6, 16), (6, 17), (7, 9), (7, 13), (8, 17), \\(9, 18), (10, 14), (11, 18), (12, 15) \}
$

\noindent
$
{\bf x} = \begin{bmatrix} 2 & -2 & 1 & -2 & 3 & -4 & -1 & 2 & 3 & 4 & 1 & -2 & 1 & -6 & 1 & 5 & -3 & 2 & -4 \end{bmatrix} 
$

\vspace{\baselineskip}
\noindent
$\boldsymbol{v_3 = 20, v_2 = 0, n=  20, m = 30}$

\noindent
$
E(G) = \{ (0, 8), (0, 11), (0, 12), (1, 9), (1, 12), (1, 13), (2, 10), (2, 16), (2, 17), (3, 11), (3, 13), (3, 14), \\(4, 11), (4, 17), (4, 18), (5, 12), (5, 15), (5, 19), (6, 13), (6, 15), (6, 16), (7, 14), (7, 17), (7, 18), (8, 14), \\(8, 15), (9, 16), (9, 19), (10, 18), (10, 19) \}
$

\noindent
$
{\bf x} = \begin{bmatrix} 1 & -3 & -1 & 1 & -2 & 2 & 2 & 3 & -4 & -1 & -1 & 1 & 3 & -2 & 1 & -2 & 4 & -3 & 2 & -1 \end{bmatrix} 
$

\vspace{\baselineskip}
\noindent
$\boldsymbol{v_3 = 20, v_2 = 2, n=  22, m = 32}$

\noindent
$
E(G) = \{ (0, 9), (0, 14), (0, 15), (1, 10), (1, 15), (1, 17), (2, 11), (2, 12), (2, 21), (3, 11), (3, 19), (3, 20), \\(4, 12), (4, 20), (5, 13), (5, 17), (5, 18), (6, 13), (6, 18), (6, 19), (7, 14), (7, 15), (7, 16), (8, 18), (8, 20),\\
 (8, 21), (9, 14), (9, 16), (10, 16), (10, 17), (12, 21), (13, 19) \}
$

\noindent
$
{\bf x} = \begin{bmatrix} 1 & 1 & 1 & -1 & 3 & -2 & 4 & -2 & -2 & 1 & 1 & 3 & 1 & -3 & -2 & 1 & 1 & -2 & 5 & -2 & -1 & -4 \end{bmatrix} 
$

\vspace{\baselineskip}
\noindent
$\boldsymbol{v_3 = 22, v_2 = 0, n=  22, m = 33}$

\noindent
$
E(G) = \{ (0, 7), (0, 12), (0, 16), (1, 8), (1, 13), (1, 14), (2, 9), (2, 14), (2, 15), (3, 10), (3, 16), (3, 17), \\(4, 11), (4, 20), (4, 21), (5, 13), (5, 15), (5, 17), (6, 18), (6, 19), (6, 20), (7, 12), (7, 16), (8, 19), (8, 21), \\(9, 14), (9, 19), (10, 17), (10, 18), (11, 20), (11, 21), (12, 18), (13, 15) \}
$

\noindent
$
{\bf x} = \begin{bmatrix} 2 & 1 & -5 & -4 & 1 & 6 & -2 & 2 & -2 & 4 & -2 & 1 & 4 & -1 & 3 & -7 & -6 & 8 & -4 & 2 & 2 & -3 \end{bmatrix} 
$

\vspace{\baselineskip}
\noindent
$\boldsymbol{v_3 = 26, v_2 = 0, n=  26, m = 39}$

\noindent
$
E(G) = \{ (0, 1), (0, 2), (0, 3), (1, 4), (1, 5), (2, 6), (2, 7), (3, 8), (3, 9), (4, 10), (4, 11), (5, 12), (5, 13),\\ (6, 14), (6, 15), (7, 8), (7, 16), (8, 17), (9, 10), (9, 18), (10, 19), (11, 12), (11, 20), (12, 13), (13, 14),\\ (14, 20), (15, 16), (15, 21), (16, 22), (17, 18), (17, 22), (18, 23), (19, 23), (19, 24), (20, 24), (21, 24), \\(21, 25), (22, 25), (23, 25) \}
$

\noindent
$
{\bf x} = \big[\begin{matrix} 2 & -1 & -2 & 3 & -1 & -1 & 2 & -4 & 1 & -3 & -1 & 2 & 2 & -1 & -1 & 3 & 1 & 1 & -2 & 4 & -1  \end{matrix} 
$

\noindent
$
\begin{matrix} -3 & 1 & 2 & -1 & -2 \end{matrix} \big]
$

\section*{Appendix $\mathbf{\Pi}$: Polyhedral seed graphs}

For each cubic polyhedral seed graph $G$, the entry lists $v_3$, $v_2$, $n$, $m$
followed by $m$ edges $E(G)$ and the kernel eigenvector $\bf x$ as
a list of integer entries.

\vspace{\baselineskip}
\noindent
$\boldsymbol{v_3 = 12, v_2 = 0, n=  12, m = 18}$

\noindent
$
E(G) = \{ (0, 4), (0, 7), (0, 8), (1, 5), (1, 7), (1, 9), (2, 6), (2, 9), (2, 11), (3, 7), (3, 10), (3, 11), (4, 8), \\ (4, 10), (5, 8), (5, 9), (6, 10), (6, 11) \}
$

\noindent
$
{\bf x} = \begin{bmatrix} 1 & 1 & 1 & -2 & 1 & -2 & 1 & 1 & -2 & 1 & 1 & -2 \end{bmatrix} 
$

\vspace{\baselineskip}
\noindent
$\boldsymbol{v_3 = 26, v_2 = 0, n=  26, m = 39}$

\noindent
$
E(G) = \{ (0, 1), (0, 2), (0, 3), (1, 4), (1, 5), (2, 6), (2, 7), (3, 8), (3, 9), (4, 10), (4, 11), (5, 12), (5, 13),\\ (6, 14), (6, 15), (7, 8), (7, 16), (8, 17), (9, 10), (9, 18), (10, 19), (11, 12), (11, 20), (12, 13), (13, 14),\\ (14, 20), (15, 16), (15, 21), (16, 22), (17, 18), (17, 22), (18, 23), (19, 23), (19, 24), (20, 24), (21, 24), \\(21, 25), (22, 25), (23, 25) \}
$

\noindent
$
{\bf x} = \big[\begin{matrix} 2 & -1 & -2 & 3 & -1 & -1 & 2 & -4 & 1 & -3 & -1 & 2 & 2 & -1 & -1 & 3 & 1 & 1 & -2 & 4 & -1 & -3   \end{matrix} 
$

\noindent
$
\begin{matrix}  1 & 2 & -1 & -2 \end{matrix} \big]
$

\vspace{\baselineskip}
\noindent
$\boldsymbol{v_3 = 28 , v_2 = 0, n=  28, m = 42}$

\noindent
$
E(G) = \{ (0, 1), (0, 2), (0, 3), (1, 4), (1, 5), (2, 3), (2, 6), (3, 7), (4, 8), (4, 9), (5, 6), (5, 10), (6, 7), (7, 11),\\ 
(8, 12), (8, 13), (9, 10), (9, 14), (10, 15), (11, 12), (11, 16), (12, 16), (13, 17), (13, 18), (14, 19), (14, 20), \\
(15, 21), (15, 22), (16, 17), (17, 18), (18, 22), (19, 22), (19, 23), (20, 24), (20, 25), (21, 23), (21, 25),\\ (23, 26), (24, 26), (24, 27), (25, 27), (26, 27) \}
$

\noindent
$
{\bf x} = \big[\begin{matrix} 2 & 3 & -7 & 4 & -4 & 2 & -6 & 5 & 2 & -5 & 3 & 2 & -1 & 5 & 1 & 3 & -4 & -1 & -1 & -2 & 7 & 1 & -4  \end{matrix} 
$

\noindent
$
\begin{matrix}  3 & 5 & -6 & 1 & -8  \end{matrix} \big]
$

\section*{Appendix N: Toroidal seed graphs}

For each  cubic toroidal seed graph $G$, the entry lists $v_3$, $v_2$, $n$, $m$
followed by $m$ edges $E(G)$ and the kernel eigenvector $\bf x$ as
a list of integer entries.

\vspace{\baselineskip}
\noindent
$\boldsymbol{v_3 = 20, v_2 = 0, n=  20, m = 30}$

\noindent
$
E(G) = \{ (0, 8), (0, 11), (0, 12), (1, 9), (1, 12), (1, 13), (2, 10), (2, 16), (2, 17), (3, 11), (3, 13), (3, 14), \\(4, 11), (4, 17), (4, 18), (5, 12), (5, 15), (5, 19), (6, 13), (6, 15), (6, 16), (7, 14), (7, 17), (7, 18), (8, 14), \\(8, 15), (9, 16), (9, 19), (10, 18), (10, 19) \}
$

\noindent
$
{\bf x} = \begin{bmatrix} 1 & -3 & -1 & 1 & -2 & 2 & 2 & 3 & -4 & -1 & -1 & 1 & 3 & -2 & 1 & -2 & 4 & -3 & 2 & -1 \end{bmatrix} 
$

\vspace{\baselineskip}
\noindent
$\boldsymbol{v_3 = 22 , v_2 = 0, n=  22, m = 33}$

\noindent
$
E(G) = \{ (0, 9), (0, 15), (0, 16), (1, 10), (1, 20), (1, 21), (2, 11), (2, 13), (2, 14), (3, 11), (3, 14), (3, 18),\\ 
(4, 12), (4, 16), (4, 17), (5, 12), (5, 16), (5, 21), (6, 13), (6, 17), (6, 18), (7, 14), (7, 17), (7, 19), (8, 18), \\
(8, 19), (8, 20), (9, 15), (9, 19), (10, 20), (10, 21), (11, 13), (12, 15) \}
$

\noindent
$
{\bf x} = \begin{bmatrix} 2 & -1 & 4 & -5 & -4 & 2 & 3 & 1 & 2 & -3 & -1 & -7 & 1 & 1 & 6 & 2 & 1 & -2 & 1 & -4 & 3 & -2  \end{bmatrix} 
$

\vspace{\baselineskip}
\noindent
$\boldsymbol{v_3 = 24 , v_2 = 0, n=  24, m = 36}$

\noindent
$
E(G) = \{ (0, 11), (0, 12), (0, 16), (1, 12), (1, 13), (1, 14), (2, 13), (2, 14), (2, 15), (3, 13), (3, 15), (3, 18),\\ 
(4, 14), (4, 16), (4, 21), (5, 15), (5, 18), (5, 19), (6, 16), (6, 17), (6, 23), (7, 17), (7, 19), (7, 20), (8, 18),\\
 (8, 20), (8, 21), (9, 19), (9, 21), (9, 22), (10, 20), (10, 22), (10, 23), (11, 22), (11, 23), (12, 17) \}
$

\noindent
$
{\bf x} = \begin{bmatrix} 1 & 1 & 1 & -2 & -2 & 1 & 1 & -2 & 1 & 1 & 1 & -2 & 1 & 1 & -2 & 1 & 1 & -2 & -2 & 1 & 1 & 1 & -2 & 1 \end{bmatrix} 
$

\end{document}